\documentclass[a4paper,12pt,reqno]{amsart}
\usepackage{amsmath}
\usepackage{amssymb}
\usepackage{amsfonts}
\usepackage{amscd}
\usepackage{amsthm}
\usepackage{mathrsfs}
\usepackage{tikz}
\usetikzlibrary{arrows, cd, positioning, shapes}
\tikzset{
    every node/.style={inner sep=2pt},
    arr/.style={thin, -{Stealth[scale=1.0]}},
    revarr/.style={thin, {Stealth[scale=1.0]}-},
    loopstyle/.style={thin, -{Stealth[scale=1.0]},
        loop left, min distance=1.5cm, out=135, in=225, distance=1.5cm
    }
}
\tikzcdset{arrow style=tikz}

\usepackage[all]{xy}
\usepackage[bb=boondox,bbscaled=.95,cal=boondoxo]{mathalfa}
\usepackage[numbers]{natbib}
\usepackage[a4paper, margin=3.5cm]{geometry}
\usepackage{hyperref}

\newcommand{\RR}{{\rm\bf R}}
\newcommand{\QQ}{{\rm\bf Q}}

\DeclareMathOperator{\SL}{\mathrm{SL}}

\DeclareMathOperator{\Aut}{\mathrm{Aut}}

\DeclareMathOperator{\Gal}{\mathrm {Gal}}

\theoremstyle{plain}
\newtheorem{theorem}{Theorem}[section]

\newtheorem{corollary}[theorem]{Corollary}
\newtheorem{proposition}[theorem]{Proposition}

\theoremstyle{remark}
\newtheorem{remark}[theorem]{Remark}
\newtheorem{definition}[theorem]{Definition}

\bibpunct{[}{]}{,}{n}{}{,}

\begin{document}

    \title{Rational quiver representations:\\tame and wild}

    \author{Fabian Januszewski}
    \address{Institut f\"ur Mathematik, Fakult\"at EIM, Paderborn University, Warburger Str.\ 100, 33098 Paderborn, Germany}
    \email{fabian.januszewski@math.uni-paderborn.de}
    \subjclass[2010]{Primary: 16G20}

    \begin{abstract}
      We study representation finite $K$-rational quivers over fields of characteristic $0$ and their indecomposable representations, exploiting that all Brauer obstructions for descent of representations are trivial in this case. Contrasting the tame case, we give an example of a simple quiver of wild representation type, where we realize every possible Brauer obstruction of a given Galois extension $L/K$ in the category of quiver representations over $L$.
    \end{abstract}

    \maketitle

    {\small
      \tableofcontents
    }

    %\addcontentsline{toc}{section}{Introduction}
    \section*{Introduction}
    \label{sec:introduction}

In \cite{januszewskirationalquivers} the author introduced the notion of rational quivers and their representations. Initially, the motivation was to gain a different perspective on the rational structures on Harish-Chandra modules previously studied by Harder \cite{harderraghuram}, Harris \cite{harris2013,harris2013erratum} and the author \cite{januszewskirationality}, which in turn were motivated by applications in number theory \cite{januszewskiperiods1,januszewskiperiods2,januszewskilocallyalgebraic}. In \cite{januszewskirationalquivers} the author showed that Gelfand's equivalence between certain blocks of Harish-Chandra modules for $\SL_2(\RR)$ over the complex numbers naturally extends to an equivalence over any field of characteristic $0$, if we furnish the Gelfand quiver with a suitable natural $\QQ$-rational structure.

In this article, we investigate various properties of $K$-rational quivers and their representations. After a quick recap of the main definitions and results of \cite{januszewskirationalquivers} in Section 1, we consider $K$-rational quivers of finite representation type in Section 2. We generalize one direction of Gabriel's celebrated $ADE$ classification result \cite{Gabriel1972,BGP1973,DlabRingel1975}: If the underlying quiver $\Gamma$ of a $K$-rational quiver $\Gamma_K$ is of type $A$, $D$ or $E$, then $\Gamma_K$ is of finite representation type (cf.~Theorem \ref{thm:rationalgabriel}).

In Remark \ref{rmk:conversefailure} we show that the converse is not true: The quiver $\Gamma$ underlying a general $K$-rational quiver $\Gamma_K$ of finite representation type is not necessarily of type $ADE$. In particular, if we assume that the rational structure on $\Gamma_K$ is given by an action of the Galois group of a finite extension $L/K$, then the base change $\Gamma_L$ of $\Gamma_K$ to $L$ need not be of finite representation type if $\Gamma_K$ was of finite representation type.

As an illustration of our generalization of Gabriel's Theorem we discuss in Section 3 a non-abelian Galois action in the context of triality and describe all indecomposable representations of the resulting $K$-rational quiver explicitly.

A crucial fact in this context is that there is no Brauer obstruction in Gabriel's situation of an $ADE$ type quiver (cf.~Corollary \ref{cor:adebrauer}). Section 4 is dedicated to the $2$-loop quiver with trivial rational structure for a given Galois extension $L/K$ in characteristic $0$. We show that every (non-trivial) class in the Brauer group arises as an obstruction for the descent of an absolutely irreducible representation of the $2$-loop quiver in the extension $L/K$.

\medskip    
%\addcontentsline{toc}{subsection}{Acknowledgements}
\noindent{\bf Acknowledgements.} 
The author acknolwedges support by the Deutsche Forschungsgemeinschaft (DFG, German Research Foundation) -- SFB-TRR 358/1 2023 -- 491392403.

\section{Rational quiver representations}

We recall the main definitions from \cite{januszewskirationalquivers}.

\subsubsection*{Rational structures on quivers}

\begin{definition}[{$K$-rational quiver, \cite{januszewskirationalquivers}}]\label{def:rationalquiver}
Let $\Gamma=(V,E,s,t)$ be a classical quiver, consisting of a finite set of vertices $V$, a finite set of edges $E$, and two maps $s,t\colon E\to V$ for the source and target of an edge. For a Galois extension $L/K$, a \emph{$K$-rational structure} on $\Gamma$ is a group homomorphism $\rho\colon\mathrm{Gal}(L/K)\to\Aut(\Gamma)$ such that the induced actions on $V$ and $E$ are continuous (for the discrete topology on $V$ and $E$). This means that for any $\sigma \in \mathrm{Gal}(L/K)$ and $e \in E$, the relations
\begin{equation}
  s(\sigma e)=\sigma s(e)\quad\text{and}\quad t(\sigma e)=\sigma t(e)
  \label{eq:rationalquiver}
\end{equation}
hold. The pair $\Gamma_K = (\Gamma, \rho)$ is called a \emph{$K$-rational quiver}.
\end{definition}

\begin{definition}[{Representation of a rational quiver, \cite{januszewskirationalquivers}}]
A \emph{$K$-rational representation} $M$ of a $K$-rational quiver $\Gamma_K$ (with respect to $L/K$) consists of:
\begin{enumerate}
    \item A family $M=(M(v))_{v\in V}$ of $L$-vector spaces.
    \item For each $\sigma\in\mathrm{Gal}(L/K)$, a family of $\sigma$-linear isomorphisms $\varphi_{v,\sigma}\colon M(v)\to M(\sigma v)$ satisfying the cocycle condition $\varphi_{\tau v,\sigma}\circ\varphi_{v,\tau}=\varphi_{v,\sigma\tau}$.
    \item A family of $L$-linear maps $\phi_e\colon M(s(e))\to M(t(e))$ for each edge $e\in E$, which are commute with the Galois action.
    %\begin{equation}
    %  \phi_{\sigma e}\circ\varphi_{s(e),\sigma}=\varphi_{t(e),\sigma}\circ\phi_e.
    %  \label{eq:rationalquiverrep}
    %\end{equation}
\end{enumerate}
\end{definition}

\subsubsection*{Associated \'etale $K$-species}

\begin{definition}[$K$-species]
A \emph{$K$-species} $S=(L_i,{}_iM_j)_{i,j\in I}$ is a finite collection $(L_i)_{i\in I}$ of division rings containing $K$ in their center, together with an $L_i,L_j$-bimodule ${}_iM_j$ for each pair $i,j\in I$, which is finite-dimensional over $K$.
\end{definition}

This classical definition was refined by the author as follows:

\begin{definition}[{\'Etale $K$-species, \cite{januszewskirationalquivers}}]
An \emph{\'etale $K$-species} is a $K$-species $S=(L_i,{}_iM_j)_{i,j\in I}$ where each $L_i$ is a finite separable extension of $K$ and each ${}_iM_j$ is a commutative $L_i,L_j$-bialgebra which is finite and \'etale as a $K$-algebra.
\end{definition}

To any $K$-rational quiver $\Gamma_K$ we can associate functorially an \'etale $K$-species $S(\Gamma_K) = (L_i, {}_iM_j)_{i,j \in I}$ as follows (cf.\ Definition 2.13 in \cite{januszewskirationalquivers}).
\begin{itemize}
    \item The index set $I$ of the species is the set of orbits of the Galois group $G = \mathrm{Gal}(L/K)$ on the set of vertices $V$, i.e., $I = G \backslash V$.
    \item For each orbit $i \in I$, choose a representative vertex $v_i \in i$. The associated division algebra $L_i$ is the fixed field $L^{G_{v_i}}$, where $G_{v_i}$ is the stabilizer of $v_i$ in $G$. This is a finite separable extension of $K$.
    \item For any two orbits $i,j \in I$, let $E_{ij} = \{e \in E \mid s(e) \in i, t(e) \in j\}$. The $L_i$-$L_j$-bimodule is defined as the direct sum over the orbits in $G \backslash E_{ij}$:
    \[ {}_iM_j := \bigoplus_{\varepsilon \in G \backslash E_{ij}} L^{G_{e_\varepsilon}}, \]
    where $e_\varepsilon$ is a representative of the edge orbit $\varepsilon$.
\end{itemize}
We showed in \cite{januszewskirationalquivers} that this construction gives rise to a categorical anti-equivalence between $K$-rational quivers and \'etale $K$-species (cf.\ Theorem 2.17 in loc.\ cit.) preserving representations, i.\,e.\ the category of representations of $\Gamma_K$ is canonically equivalent to the category of representations of the associated $K$-species $S(\Gamma_K)$ (cf.\ Theorem 2.35 in loc.\ cit.)

\section{$K$-rational quivers of finite representation type}

Let $\Gamma$ be a connected quiver of type $A$, $D$ or $E$. Let $L/K$ denote a Galois extension of fields of characteristic $0$ and let $L^{\rm alg}$ denote an algebraic closure of $L$. By Gabriel's Theorem (cf.\ \cite{Gabriel1972,BGP1973,DlabRingel1975}), the indecomposable representations of $\Gamma$ over $L^{\rm alg}$ are in bijection with set $\Delta^+$ of positive roots of the root system $\Delta$ associated to the Dynkin diagram underlying $\Gamma$. Moreover, the endomorphism ring of each indecomposable representation $V$ is known to isomorphic to $L^{\rm alg}$.

The explicit description of the indecomposable representation via dimension vectors shows that each indecomposable representation of $\Gamma$ descends to a representation over $L$ and by Hilbert 90, this descent is unique up to isomorphim. Now every indecomposable representation $V$ of $\Gamma$ over $L$ decomposes over $L^{\rm alg}$ into a direct sum of indecomposable representations, which each descend to $L$, hence $V$ must agree with one of them. This shows

\begin{theorem}\label{thm:splitclassification}
  Let $\Gamma$ be a connected quiver of type $A$, $D$, or $E$. Then the indecomposable representaitons of $\Gamma$ over any field $L$ of charactereristic $0$ are in canonical bijection with the positive roots of the underyling root system $\Delta$. The endomorphism rings of each indecomposable representation over $L$ is $L$ and each irreducible representation over $L$ is absolutely irreducible.
\end{theorem}

\begin{proof}
  The main statement was proven in the preceeding discussion. The remaining claims follow with Proposition 1.21 in \cite{januszewskirationalquivers}.
\end{proof}

\begin{corollary}\label{cor:adebrauer}
  Let $\Gamma_K$ denote a connected quiver over $K$ which is of type $A$, $D$ or $E$, which splits inside a Galois extension $L/K$. Then there are no Brauer obstructions for descent of irreducibles for intermediate fields inside the extension $L/K$.
\end{corollary}

Combining Theorem \ref{thm:splitclassification} with Corollary \ref{cor:adebrauer} we obtain

\begin{theorem}\label{thm:rationalgabriel}
  Let $\Gamma_K$ denote a quiver over $K$ whose underlying quiver is of $\Gamma$ is of type $A$, $D$ or $E$. Then $\Gamma_K$ is representation finite and the set of isomorphism classes of indecomposable representations over $K$ is in canonical bijection with the set of $\Gal(L/K)$-orbits acting on the set $\Delta_+\subseteq\Delta$ of positive roots.
\end{theorem}

The Galois action in the statement is the canonical extension to the root system $\Delta$ of the given action of $\Gal(L/K)$ on the underlying Dynkin diagram.

\begin{proof}
  Assume first that $\Gamma$ is of type $ADE$. The classification of indecomposables for $\Gamma_K$ is a straightforward application of Galois descent: Let $\mathcal O\subseteq\Delta_+$ denote a Galois orbit. Fix a representative $\alpha\in\mathcal O$, write $G_\alpha\subseteq\Gal(L/K)$ for its stabilizer. Then the indecomposable representation associated to $\alpha$ is defined over $L^{G_\alpha}$ and its restriction of scalars inside the extension $L^{G_\alpha}$ to $K$ is easily seen to be an indecomposable representation of $\Gamma_K$ over $K$, independent of the choice of representative $\alpha$. Moreover, every indecomposable representation over $K$ is of this form, again by the same argument as above.
\end{proof}

\begin{remark}
    We remark that the implication $\Gamma$ of type $ADE$ implies $\Gamma_K$ is of finite representation type also is a consequence of Dlab--Ringel's extension of Gabriel's Theorem to $K$-species (cf.~\cite{DlabRingel1975}) combined with Theorem 2.35 in \cite{januszewskirationalquivers}: The category of representations of $\Gamma_K$ is equivalent to that of the associated $K$-species $S(\Gamma_K)$. By Dlab--Ringel, this $K$-species is of finite representation type if and only if its associated valued graph is a Dynkin diagram (i.\,e.\ of type $A, B, C, D, E, F$, or $G$). Considering the possible autorphisms of diagrams of types $A$, $D$ and $E$ and the resulting foldings, the folding process can never create a graph of Euclidean (affine) or wild type, for the quiver structure on $\Gamma$ imposes additional constraints (e.\,g.\ the central arrow of a quiver of type $A_{2n}$ forces the Galois action to be trivial etc.). Consequently, the species $S(\Gamma_K)$ is seen to have a Dynkin diagram as valued graph.
\end{remark}

\begin{remark}[{A finite type $K$-rational quiver of infinite type over $L$}]\label{rmk:conversefailure}
  There is no converse to Theorem \ref{thm:rationalgabriel} in the spirit of Gabriel's $ADE$ characterization. Consider the Kronecker quiver $\Gamma$, which has two vertices $v_1, v_2$ and two parallel edges $a, b: v_1 \to v_2$. This quiver is of tame infinite representation type over any field. Let $L/K$ be a quadratic extension with Galois group $G=\{1, \tau\}$. We define a $K$-rational quiver $\Gamma_K$ via a $K$-rational structure on $\Gamma$ by letting $\tau$ act trivially on the vertices and by permuting the edges: $\tau(v_i) = v_i$, $\tau(a)=b$, and $\tau(b)=a$. The associated \'etale $K$-species $S(\Gamma_K)$ has two nodes, corresponding to the two fixed vertex orbits $\{v_1\}$ and $\{v_2\}$, with associated field $K$ for both. The two edges form a single orbit whose stabilizer is trivial, yielding a $(K,K)$-bimodule $L$. The resulting valued graph is that of type $B_2$:
\begin{center}
\begin{tikzpicture}[node distance=4cm, auto]
    \node (K1) {$K$};
    \node (K2) [right=of K1] {$K$};
    \draw[-{Stealth[scale=1.0]}] (K1) to node[above]{$L$} node[below]{(2,2)} (K2);
\end{tikzpicture}
\end{center}
By the Dlab--Ringel Theorem \cite{DlabRingel1975}, this species is of finite representation type. Thus, the rational quiver $\Gamma_K$ is of finite type over $K$, despite its underlying quiver $\Gamma$ being of infinite type over $L$.
\end{remark}

\section{A rational quiver of type $D_4$ and triality}

We illustrate explore Theorem \ref{thm:rationalgabriel} with an example involving a non-abelian Galois group acting on a quiver of finite representation type in the context of triality: We consider a quiver of type $D_4$ and a rational structure based on a Galois extension with group $S_3$ and classify the indecomposable representations.

\subsubsection*{The $K$-rational quiver}
Let $K$ be a field with characteristic $0$. Let $L/K$ be a Galois extension with Galois group $G = \mathrm{Gal}(L/K) \cong S_3$, the symmetric group on three letters.

We consider the quiver $\Gamma$ with underlying graph the Dynkin diagram $D_4$. We label the central \lq{}root\rq{} vertex $v_0$ and the outer \lq{}leaf\rq{} vertices $v_1, v_2, v_3$. We orient the quiver by having all arrows point outwards from the center, with three edges $e_i: v_0 \to v_i$ for $i=1,2,3$.

We define a $K$-rational structure on $\Gamma$ by letting $S_3$ act naturally on the indices $\{1,2,3\}$. For any permutation $\sigma \in S_3$:
\begin{itemize}
    \item Action on vertices: $\sigma v_0 = v_0$ and $\sigma v_i = v_{\sigma^{-1}(i)}$ for $i \in \{1,2,3\}$.
    \item Action on edges: $\sigma e_i = e_{\sigma^{-1}(i)}$ for $i \in \{1,2,3\}$.
\end{itemize}
This action $\rho_{S_3}$ on $\Gamma$ gives rise to a $K$-rational quiver $\Gamma_K = (\Gamma, \rho_{S_3})$ in the sense of Definition \ref{def:rationalquiver}.

\subsubsection*{The associated \'etale $K$-species}
The associated $K$-species $S(\Gamma_K)$ is given by the following data.
\begin{enumerate}
    \item \textbf{The index set and fields:} The two vertex orbits are $i_0 = \{v_0\}$ and $i_1 = \{v_1, v_2, v_3\}$.
    \begin{itemize}
        \item For $i_0$, the stabilizer is $S_3$, so the field is $L_{i_0} = L^{S_3} = K$.
        \item For $i_1$, the stabilizer of the representative $v_1$ is $H \cong S_2$. The field is $L_{i_1} = L^H = M$.
    \end{itemize}
    \item \textbf{The bimodule:} The single edge orbit $\{e_1, e_2, e_3\}$ connects $i_0$ to $i_1$. The stabilizer of the representative $e_1$ is $H$. The resulting $(K,M)$-bimodule is ${}_{i_0}M_{i_1} = L^H = M$.
\end{enumerate}
The resulting \'etale $K$-species $S(\Gamma_K)$ is of type $G_2$:
\begin{center}
\begin{tikzpicture}[node distance=4cm, auto]
    \node (K_node) {$K$};
    \node (M_node) [right=of K_node] {$M$};
    \draw[-{Stealth[scale=1.0]}] (K_node) to node[above]{${}_K M_M = M$} (M_node);
\end{tikzpicture}
\end{center}
Since the species is of Dynkin type $G_2$, it has 6 indecomposable representations, corresponding to the 6 positive roots of $G_2$. This matches the number of Galois orbits of the 12 positive roots of $D_4$, see below.

\subsubsection*{The Galois action on the $D_4$ root system}

In order to apply Theorem \ref{thm:rationalgabriel} in the given situation, we first describe the Galois action on the root system $D_4$.

We denote the simple roots corresponding to the vertices $\{v_0, v_1, v_2, v_3\}$ by $\{\alpha_0, \alpha_1, \alpha_2, \alpha_3\}$. The action of an element $\sigma \in S_3$ on the quiver induces an action on these simple roots:
\[ \sigma(\alpha_0) = \alpha_0 \quad \text{and} \quad \sigma(\alpha_i) = \alpha_{\sigma^{-1}(i)} \text{ for } i \in \{1,2,3\}. \]
This action extends to the entire set of $12$ positive roots. For clarity, we will identify each positive root with the dimension vector of its corresponding indecomposable representation, written in the basis $(v_0, v_1, v_2, v_3)$. The action of $\sigma$ sends a representation with dimension vector $(d_0, d_1, d_2, d_3)$ to one with dimension vector $(d_0, d_{\sigma^{-1}(1)}, d_{\sigma^{-1}(2)}, d_{\sigma^{-1}(3)})$.

The 12 positive roots (dimension vectors) of $D_4$ are:
\begin{itemize}
    \item $(1,0,0,0)$, $(0,1,0,0)$, $(0,0,1,0)$, $(0,0,0,1)$
    \item $(1,1,0,0)$, $(1,0,1,0)$, $(1,0,0,1)$
    \item $(1,1,1,0)$, $(1,1,0,1)$, $(1,0,1,1)$
    \item $(1,1,1,1)$
    \item $(2,1,1,1)$
\end{itemize}

Applying the permutation action of $S_3$ to these $12$ dimension vectors partitions them into exactly $6$ orbits:

\begin{itemize}
    \item \textbf{Orbit 1} (size $1$)\textbf{:} $(1,0,0,0)$ is fixed by the action.
    \item \textbf{Orbit 2} (size $3$)\textbf{:} $\{(0,1,0,0), (0,0,1,0), (0,0,0,1)\}$ are permuted.
    \item \textbf{Orbit 3} (size $3$)\textbf{:} $\{(1,1,0,0), (1,0,1,0), (1,0,0,1)\}$ are permuted.
    \item \textbf{Orbit 4} (size $3$)\textbf{:} $\{(1,1,1,0), (1,1,0,1), (1,0,1,1)\}$ are permuted.
    \item \textbf{Orbit 5} (size $1$)\textbf{:} $(1,1,1,1)$ is fixed.
    \item \textbf{Orbit 6} (size $1$)\textbf{:} $(2,1,1,1)$ is fixed.
\end{itemize}

This partitioning leads by Theorem \ref{thm:rationalgabriel} directly to the main classification result for the $K$-rational triality quiver.

\begin{theorem}
The $K$-rational quiver $\Gamma_K$ of type $D_4$ with $S_3$-action has exactly $6$ indecomposable representations over the field $K$. Each indecomposable corresponds to one of the $6$ Galois orbits of the $12$ positive roots of $D_4$.
\end{theorem}

These $6$ indecomposable $K$-rational representations are described as follows:
\begin{enumerate}
    \item \textbf{The simple representation at $v_0$:} Corresponds to Orbit 1. It is the absolutely irreducible representation $M_K$ over $K$ with dimension vector $(1,0,0,0)$, realized as $M(v_0)=L$ and $M(v_i)=0$ for $i=1,2,3$ with the tautological Galois actions.
    \item \textbf{The indecomposable from the simple leaves:} Corresponds to Orbit 2. Its base change to $L$ is the direct sum of the three simple representations at the leaf vertices. Over $K$ it is the direct sum with the appropriate Galois action permuting the three summands.
    \item \textbf{The indecomposable from the projective leaves:} Corresponds to Orbit 3. Its base change to $L$ is the direct sum of the three representations with dimension vectors $(1,1,0,0), (1,0,1,0), (1,0,0,1)$.
    \item \textbf{The indecomposable of type $(1,1,1,0)$:} Corresponds to Orbit 4. Its base change to $L$ is the direct sum of the three representations in that orbit.
    \item \textbf{The projective indecomposable at $v_0$:} Corresponds to Orbit 5. It is the abolutely irreducible representation over $K$ with dimension vector $(1,1,1,1)$.
    \item \textbf{The injective indecomposable at $v_0$:} Corresponds to Orbit 6. It is the absolutely irreducible representation with dimension vector $(2,1,1,1)$.
\end{enumerate}
This completes the classification.

\subsubsection*{The classification for the associated species}
A representation of $S(\Gamma_K)$ is a pair $(V_K, V_M)$ of a $K$-vector space and an $M$-vector space, with a $K$-linear map $\hat{f}: V_K \to V_M$. In this notation, the indecomposable representations are:
\begin{enumerate}
    \item \textbf{Simple at $K$-node:} $(K, 0)$ with $\hat{f} = 0$. Corresponds to the simple representation at $v_0$ associated to orbit 1.
    \item \textbf{Simple at $M$-node:} $(0, M)$ with $\hat{f} = 0$. Corresponds to orbit 2 for the simple representations at $\{v_1, v_2, v_3\}$.
    \item \textbf{Projective cover of simple at $K$:} $(K, M)$ with $\hat{f}: K \hookrightarrow M$ the canonical inclusion. Corresponds to the orbit of representations with dimension vector $(1,1,0,0)$ and its permutations (orbit 3).
    \item \textbf{Injective hull of simple at $M$:} $(M, M)$ with $\hat{f} = \mathrm{id}_M$. Corresponds to the orbit of representations with dimension vector $(1,1,1,0)$ and its permutations (orbit 4).
    \item \textbf{Intermediate indecomposable:} Let $\ker(\mathrm{Tr}_{M/K})$ be the 2-dim $K$-subspace of $M$ of trace-zero elements. The representation is $(\ker(\mathrm{Tr}_{M/K}), M)$ with $\hat{f}$ being the inclusion. Corresponds to the Galois-fixed representation with dimension vector $(2,1,1,1)$ (orbit 5).
    \item \textbf{Injective hull of simple at $K$:} $(M, K)$ with $\hat{f} = \mathrm{Tr}_{M/K}$. Corresponds to the Galois-fixed projective indecomposable at $v_0$ with dimension vector $(1,1,1,1)$ (orbit 6).
\end{enumerate}
As expected, the 6 indecomposable rational representations of $\Gamma_K$ are in perfect correspondence with the 6 indecomposable representations of its associated \'etale $K$-species of type $G_2$.

\section{The wild case: Realizing arbitrary Brauer obstructions}

Let $K$ denote a field of characteristic $0$. Let $D$ be a central division algebra over $K$ of degree $n$, and let $L/K$ be a Galois splitting field for $D$ of degree $n$ with Galois group $\Gal(L/K)$. We fix a $\Gal(L/K)$-equivariant $L$-algebra isomorphism $\Phi\colon L \otimes_K D \stackrel{\cong}{\longrightarrow} M_n(L)$. By Albert's Theorem, we may choose two elements $d_1, d_2 \in D$ that generate $D$ as a $K$-algebra (see, for instance, \cite[Prop.~11.2]{KMRT98}).

We give a construction that realizes $D$ as the endomorphism ring of a $K$-rational representation of two $2$-loop quiver with trivial rational structure.

%\subsubsection*{The $K$-rational quiver}

\begin{itemize}
    \item \textbf{Quiver $\Gamma$:} A single vertex $v$ and two loops, $e_1$ and $e_2$, at $v$:
\begin{center}
\begin{tikzpicture}
    \node (v) {$v$};
    \draw[-{Stealth[scale=1.0]}] (v) edge[loop left] node[left] {$e_1$} (v);
    \draw[-{Stealth[scale=1.0]}] (v) edge[loop right] node[right] {$e_2$} (v);
\end{tikzpicture}
\end{center}
    \item \textbf{$K$-rational structure:} We endow $\Gamma$ with the trivial rational structure for $L/K$.
\end{itemize}

\begin{remark}
The choice of a quiver with one vertex and two loops is crucial. Its path algebra over $K$ is the free algebra $K\langle e_1, e_2 \rangle$, and it is of wild representation type. The wildness implies that for any finite-dimensional $K$-algebra $A$---including our division algebra $D$---there exists some representation of the quiver whose image algebra is $A$.

However, this existence result from representation theory is too abstract, since we need more information about the representation.
\end{remark}

%\subsubsection*{The irreducible representation and its $K$-rational structure}

We define an absolutely irreducible representation $M$ of $\Gamma$ over the field $L$ as follows. Put $M(v) = L^n$ and define the maps for the loops via the splitting isomorphism $\Phi$:
    \[
    \phi_{e_1} := \Phi(1 \otimes d_1) \in M_n(L) \quad \text{and} \quad \phi_{e_2} := \Phi(1 \otimes d_2) \in M_n(L),
    \]
    considered as endomorphisms of $M(v)=L^n$.

\begin{proposition}
The representation $M$ is absolutely irreducible over $L$.
\end{proposition}

\begin{proof}
Since $d_1$ and $d_2$ generate $D$ as a $K$-algebra, the matrices $\phi_{e_1}$ and $\phi_{e_2}$ generate the $L$-algebra $\Phi(L \otimes_K D) = M_n(L)$. An invariant subspace of $L^n$ is therefore an $M_n(L)$-submodule. As the standard module $L^n$ is irreducible over $M_n(L)$, $M$ is an irreducible representation. This argument is valid over any extension of $L$, so $M$ is absolutely irreducible.
\end{proof}

%\subsubsection*{The endomorphism ring and the Brauer obstruction}

Let $M_K$ be the $K$-rational representation corresponding to $M$ obtained by restriction of scalars, i.\,e.\ we consider $M(v)$ as $K$-vector space $M_K(v)$ and $\phi_{e_i}$ as $K$-linear endomorphisms of $M_K(v)$.

\begin{theorem}
  The endomorphism ring of $M_K$ in the category of $K$-rational representations is the opposite algebra of $D$:
\[
\mathrm{End}_{\Gamma_K}(M_K) \cong D^{\mathrm{op}}.
\]
\end{theorem}

\begin{proof}
  The $K$-rational representation $M_K$ corresponds to the pair
  \[
  (M, (\varphi_\sigma)_{\sigma\in \Gal(L/K)}),
  \]
  where the semi-linear maps $\varphi_\sigma$ are the standard entry-wise action of $\Gal(L/K)$ on $L^n$. An endomorphism of $M_K$ is an endomorphism of the underlying module over the $K$-path algebra $P_K(\Gamma)$. The endomorphism ring is the centralizer of the image of this path algebra in $\mathrm{End}_K(L^n)$.

Let $A = K\langle \phi_{e_1}, \phi_{e_2} \rangle \subset M_n(L)$ be the $K$-algebra generated by the representation maps. By construction, $A \cong D$. The endomorphism ring is the centralizer of this simple subalgebra:
\[
\mathrm{End}_{\Gamma_K}(M_K) = C_{\mathrm{End}_K(L^n)}(A).
\]
By the Double Centralizer Theorem, this centralizer is isomorphic to $A^{\mathrm{op}}\cong D^{\mathrm{op}}$, hence $\mathrm{End}_{\Gamma_K}(M_K) \cong D^{\mathrm{op}}$.
\end{proof}

This construction realizes the Brauer obstruction given by $D$: $M$ does not descend to $K$ unless $D=K$, the corresponding obstruction is $[D]\in H^2(\Gal(L/K);L^\times)$.

\end{document}